\documentclass[a4paper,11pt]{amsart}

\newcommand{\SFCC}{\sigma\mbox{-}\mathrm{FIN}\mbox{-}\mathrm{CC}}

\newcommand{\mf}[1]{\mathfrak{#1}}

\newcommand{\dom}[1]{\mathrm{dom}(#1)}

\let\oldtocsection=\tocsection
\let\oldtocsubsection=\tocsubsection
\let\oldtocsubsubsection=\tocsubsubsection
 
\renewcommand{\tocsection}[2]{\hspace{0em}\oldtocsection{#1}{#2}\bfseries}
\renewcommand{\tocsubsection}[2]{\hspace{1.8em}\oldtocsubsection{#1}{#2}}
\renewcommand{\tocsubsubsection}[2]{\hspace{4.4em}\oldtocsubsubsection{#1}{#2}}
\makeatletter
\renewcommand\subsection{\@startsection{subsection}{2}%
  \z@{-.5\linespacing\@plus-.7\linespacing}{.5\linespacing}%
  {\normalfont\scshape}}
\renewcommand\subsubsection{\@startsection{subsubsection}{3}%
  \z@{.5\linespacing\@plus.7\linespacing}{.5\linespacing}%
  {\normalfont\scshape}}

\makeatother
\usepackage[framemethod=tikz]{mdframed}

\usepackage[pdftex]{hyperref}
\hypersetup{
    colorlinks=true, %set true if you want colored links
    linktoc=all,     %set to all if you want both sections and subsections linked
    linkcolor=blue,  %choose some color if you want links to stand out
    allcolors=blue,
}

\usepackage[utf8]{inputenc}
\usepackage[T1]{fontenc}
\usepackage{tikz-cd}
\usepackage{paralist}
\usepackage{mathrsfs}
\usepackage{amssymb}
\usepackage{amsmath}
\usepackage{amsfonts}
\usepackage{enumitem}

\newtheorem{theorem}{Theorem}[section]
\newtheorem{corollary}[theorem]{Corollary}
\newtheorem{lemma}[theorem]{Lemma}

\newtheorem{Fact}[theorem]{Fact}

\theoremstyle{definition}
\newtheorem{definition}[theorem]{Definition}

\theoremstyle{remark}
\newtheorem{remark}[theorem]{Remark}

\theoremstyle{question}
\newtheorem{question}[theorem]{Question}

\usepackage[pdftex]{hyperref}
\hypersetup{
    colorlinks=true, %set true if you want colored links
    linktoc=all,     %set to all if you want both sections and subsections linked
    linkcolor=blue,  %choose some color if you want links to stand out
    allcolors=blue,
}

\newcommand{\brm}{\begin{remark}\begin{rm}}
\newcommand{\erm}{\end{rm}\end{remark}}

\newcommand{\bce}{\begin{compactenum}}
\newcommand{\ece}{\end{compactenum}}

\newcommand{\bb}[1]{\mathbb{#1}}
\newcommand{\cl}[1]{\mathcal{#1}}

\newcommand{\set}[2]{\{#1 \: ; \: #2\}}
\newcommand{\seq}[2]{\la#1 \: ; \: #2\ra}

\newcommand{\N}{\bb{N}}
\newcommand{\M}{\bb{M}}

\newcommand{\B}{\cl{B}}

\newcommand{\ZFC}{{\sf ZFC}}
\newcommand{\ZF}{{\sf ZF}}
\newcommand{\CH}{{\sf CH}}

\newcommand{\PFA}{{\sf PFA}}
\newcommand{\RC}{{\sf RC}}

\newcommand{\BA}{{\sf BA}(\omega_1)}

\newcommand{\restr}[2]{\ensuremath{#1 \! \upharpoonright \! #2}}

\newcommand{\SH}{{\sf SH}(\omega_1)}

\newcommand{\WP}{{\sf WC}(\omega_1)}

\newcommand{\MA}{{\sf MA}_{\omega_1}}

\newcommand{\dense}{\omega_1\mbox{-dense}}

\renewcommand{\P}{\bb{P}}

\newcommand{\sub}{\subseteq}

\newcommand{\la}{\langle}
\newcommand{\ra}{\rangle}

\newcommand{\R}{\mathbb{R}}
\newcommand{\Q}{\mathbb{Q}}
\newcommand{\Z}{\mathbb{Z}}
\newcommand{\Add}{\mathrm{Add}}

\newcommand{\BM}{(\mathcal{B}, \bar{\mu})}
\newcommand{\BMD}{\mathcal{B}}
\newcommand{\Bk}{\mathcal{B}_\kappa}

\newcommand{\x}{\times}
\newcommand{\T}{\mathbb{T}}

\newcommand{\covN}{\mathrm{cov}(\mathcal{N})}
\newcommand{\nonN}{\mathrm{non}(\mathcal{N})}

\begin{document}

\title[Rado's Conjecture and the random algebra]{Rado's Conjecture and the random algebra}

\author{Radek Honzik}
\address{Charles University, Department of Logic,
Celetn{\' a} 20, Prague~1, 
116 42, Czech Republic
}
\email{radek.honzik@ff.cuni.cz}
\urladdr{logika.ff.cuni.cz/radek}

%\author{SS}

%\thanks{
%The author was supported by GA{\v C}R grant ``The role of set theory in modern mathematics'' (24-12141S)}

\begin{abstract} 
Rado's Conjecture, $\RC$, is a compactness principle for a certain class of partial orders, namely trees $T$ of height $\omega_1$ without cofinal branches, postulating that a partial order $P$ from this class can be decomposed into at most countably many antichains if and only if all its suborders of size $\omega_1$ can be decomposed into at most countably many antichains. Rado's Conjecture is thus an uncountable version of Mirsky's theorem asserting that for every natural number $n$, every infinite partial order $P$ can be decomposed into at most $n$ many antichains if and only if all its finite  suborders can be decomposed into at most $n$ many antichains. Todorcevi{\'c} showed in \cite{T:Rado} that $\RC$ is consistent modulo a strongly compact cardinal. $\RC$ implies $2^\omega \le \omega_2$, and has powerful consequences such as the Singular Cardinal Hypothesis, the failure of $\square(\kappa)$ for every regular $\kappa \ge \omega_2$ (and hence in particular the Projective Determinacy), and the Strong Chang Conjecture. It is also known that it is incompatible with Martin Axiom, $\MA$. We show that $\RC$ is consistent with $2^\omega = \omega_2$ and the cardinal invariants in Cicho{\' n} diagram corresponding to forcing with the random algebra, i.e., $\mf{d} = \omega_1$, $\covN = \omega_2$, $\nonN = \omega_1$. This provides a new pattern of cardinal invariants known to be consistent with $\RC$. To prove the theorem, we first observe that random algebras do not specialize non-special trees of height $\omega_1$. Then we use the random algebra $\Bk$ for a strongly compact $\kappa$ to define a new version of Mitchell forcing which yields the required result.
\end{abstract}

\maketitle

\tableofcontents

\section{Introduction}
Before reviewing Rado's Conjecture, let us state a few facts about trees of height $\omega_1$ which have turned out to be the right class of partial orders to consider in this context. Let $(T,<)$ be a tree of height $\omega_1$ of an arbitrary size. We will be interested in trees $T$ which do not have cofinal branches (if $T$ is an $\omega_1$-tree, then such trees are called Aronszajn). A strengthening of the property of not having cofinal branches is the property of being \emph{special}:

\begin{definition}\label{def:special}
If $(T,<)$ is a tree of height $\omega_1$, then $T$ is called \emph{special} if either of the following equivalent conditions holds:
\bce[(i)]
\item
There is a function $f: T \to \omega$ which is 1-1 on chains, i.e., if for all $t,s \in T$, if $t <_T s$, then $f(t) \neq f(s)$.
\item
$T$ can be decomposed into at most countably many antichains.
\item
There is a homomorphism $f: (T,<) \to (\Q,<)$, where $(\Q,<)$ is the standard linear order on the rationals (we say that $T$ is $\Q$-embeddable).
\ece
\end{definition}

The less obvious implications (i)$\to$(ii) or (i)$\to$(iii) are due to Kurepa (see \cite[p.~284]{TODtophandbook} for a proof). It is of some interest to note that the equivalences are true for any partial order $(P,<)$, not necessarily a tree.

\begin{definition}\label{def:RC}
\emph{Rado's Conjecture}, $\RC$, denotes the statement that for every tree $T$ of height $\omega_1$ the following two conditions are equivalent:
\bce[(i)]
\item $T$ is special.
\item Every subtree $T$ of size $\omega_1$ is special.\footnote{The subtrees can be without loss of generality required to be closed downwards in the tree order $(T,<)$ because $T$ has height $\omega_1$.}
\ece
\end{definition}

Using the equivalences in Definition \ref{def:special}, $\RC$ is equivalent to a principle which asserts that for every tree  $(T,<)$  of height $\omega_1$, $T$ can be decomposed into at most countably many antichains if and only if every subtree of $T$ of size $\omega_1$ can be decomposed into at most countably many antichains. $\RC$ is thus an uncountable version, restricted to tree orders, of Mirsky's theorem asserting that for every partial order $(P,<)$ and any natural number $n \in \omega$, $P$ can be decomposed into at most $n$ many antichains if and only if every finite suborder can be decomposed into at most $n$ many antichains.

There are also dual versions of these principles which require decompositions into chains instead of antichains. The dual version of Mirsky's theorem for partial orders is Dilworth's theorem; interestingly, while Mirsky's theorem is provable  in set theory without the Axiom of Choice ($\ZF$), Dilworth's theorem requires the compactness principle for first-order logic and is thus really a ``compactness-type'' theorem. The dual version of Rado's Conjecture is Galvin's Conjecture. Unlike Rado's Conjecture, Galvin's Conjecture is formulated for all partial orders (not just trees) and it is open whether it is consistent (from any large cardinal). It is also easy to observe that Galvin's Conjecture implies Rado's Conjecture. 
 
Both $\RC$ and Galvin's Conjecture were originally formulated as compactness principles for certain classes of graphs, related to countable chromatic numbers. See \cite{Tdich}  and \cite{RH:comp} for more historical context and details on these conjectures.

Todorcevi{\'c} showed in \cite{T:Rado} that $\RC$ is consistent (with $\CH$) from a strongly compact cardinal using a Levy collapse, and Zhang \cite[Section 2.1]{Zhang} later showed that a standard Mitchell forcing yields the consistency of $\RC$ with $2^\omega =\omega_2$ (also from a strongly compact cardinal). Rado's Conjecture has powerful implications: $\RC$ for instance implies 
$2^\omega \le \omega_2$, and more generally $\theta^\omega = \theta$ for all regular $\theta \ge \omega_2$; The Singular Cardinal Hypothesis; For any regular cardinal $\theta \ge \omega_2$, the stationary set $\theta \cap \mathrm{cof}(\omega)$ reflects; $\square(\kappa)$ fails for all regular $\kappa \ge \omega_2$; The Strong Chang's Conjecture (see \cite{Tdich} for more details). Morever, $\RC + 2^\omega = \omega_2$ implies the strong tree property (see \cite{TPW:cc, TPW:more}), thus unifying certain conceptually different compactness principles.

It is known that $\RC$ is  incompatible with Martin Axiom, $\MA$: There are always non-special trees $T$ of height $\omega_1$ and size $2^\omega$ without cofinal branches,\footnote{For instance the tree denoted $\sigma(\R)$ composed of countable bounded subsets of the reals well-ordered by the natural linear order on the reals, ordered by the end-extension. $\sigma(\R)$ is a non-special tree of height $\omega_1$ of size $2^\omega$ without cofinal branches. Note that while $\sigma(\R)$ is not $\Q$-embeddable, it is trivially $\R$-embeddable (by the identity function). See \cite[Example 7]{T:Rado} and \cite[Observation 3.2]{Zhang}.} and $\RC$ thus implies that there must be non-special subtrees of $T$ of size and height $\omega_1$. This contradicts a consequence of $\MA$ that all trees of height and size $\omega_1$ without cofinal branches are special ($\mathsf{SAT}$).\footnote{$\RC$ can be seen as a maximalist form of $\mathsf{SAT}$ postulated for trees of all sizes: a tree of height $\omega_1$ of any size is  special exactly when all its subtrees of size $\omega_1$ are special.}

Since $\RC$ is incompatible with $\MA$, it can be viewed as a powerful alternative to forcing axioms. To evaluate its relative power to decide various mathematical statements, it would be beneficial to have a rich variety of models of $\RC$. However, at the moment, only two models are known in the literature: the Levy collapse from \cite{T:Rado} yielding $\CH$, and the Mitchell collapse from \cite{Zhang} yielding $2^\omega = \omega_2$.\footnote{Zhang \cite{Zhang} introduced a Baire version of $\RC$, $\RC^B$, which deals with trees which are $\sigma$-distributive as forcing notions (he also showed that $\RC^B$ is strictly weaker than $\RC$). He further observed that a relatively large class of forcing notions can be used to force $\RC^B$.} We will focus on $\RC + 2^\omega = \omega_2$ in this article because we wish to discuss the compatibility of $\RC$ with certain (non-trivial) patterns of the cardinal invariants of the Baire space $\omega^\omega$, and thus require the $\neg \CH$ context.

There are several reasons why forcing $\RC + 2^\omega = \omega_2$ is complicated: (i) $\RC$ is destroyed by adding a new real, and thus forcing $\RC$ together with a desired statement must be done in one step if new reals are added, and  (ii) a key step of arguments for obtaining $\RC$ is showing that a given forcing notion $\P$ does not specialize trees of height $\omega_1$ without cofinal branches. In contrast to the property of not adding new cofinal branches (a property used to show various compactness principles related to trees), the property of not specializing trees is harder to ensure.

Todorcevi{\'c} \cite{T:Rado} observed that $\sigma$-closed forcings do not specialize trees, and Zhang \cite{Zhang} observed that Cohen forcing adding any number of subsets of $\omega$ does not specialize trees either.
In this article, we observe that the random algebra $\B_\kappa$, for any infinite $\kappa$, is another example of a forcing notion which does not specialize non-special trees of height $\omega_1$; more generally, we show that all $\sigma$-finite-cc forcings do not specialize trees of height $\omega_1$ (Theorem \ref{th:random1}). We use this observation to define a ``randomized'' Mitchell forcing $\M^R_\kappa$ to force $\RC$ with different properties than the standard Mitchell forcing. By way of example we show in Theorem \ref{th:random2} that $\RC$ is consistent with $2^\omega = \omega_2$ and cardinal invariants of Cicho{\' n} diagram corresponding to random forcing.

Regarding other consequences of $\RC$, in \cite{RH:comp} we checked that $\RC$ does not decide some of the well-known mathematical problems the way $\PFA$ does, by observing that in standard Mitchell models these problems have the same truth value as in $V = L$ (we specifically discussed Whitehead's Conjecture, Suslin Hypothesis, and Baumgartner's Axiom). In Lemma \ref{th:random3} we observe that the same holds for the ``randomized'' Mitchell forcing as regards the Suslin Hypothesis and Baumgartner's axiom. The status of Whitehead's Conjecture is left open (see Question \ref{q}). It is completely open whether $\RC + 2^\omega = \omega_2$ is consistent with the way $\PFA$ decides these statements (see Question \ref{q1}). 

\medskip

\textbf{Acknowledgement.} The author wishes to thank (alphabetically) to {\v S}{\'a}rka Stejskalov{\'a} and Corey Switzer for helpful discussions of the concepts related to random algebras.

\section{Random algebras}

Let us briefly summarize basic facts regarding the random algebra $\B_\kappa$, where $\kappa$ is an infinite cardinal.

\subsection{Maharam theorem}

\begin{definition}\label{def:ma}
A \emph{probability measure algebra} is a pair $\BM$ such that $\BMD$ is a $\sigma$-complete Boolean algebra and $\bar{\mu}: \BMD \to [0,1]$ is a function which satisfies:
\bce[(i)]
\item $\bar{\mu}(0) = 0$, $\bar{\mu}(a) > 0$ for all $a \neq 0$.
\item Whenever $\seq{a_k}{k \in \N}$ is a disjoint sequence in $\BMD$, then $\bar{\mu}(\bigvee \set{a_k}{k\in \N}) = \sum_{k \in \N}\bar{\mu}(a_k)$.
\item (Probability algebra) $\bar{\mu}(1) = 1$.
\ece
\end{definition}

Note that all probability measure algebras are ccc and complete (by ccc and $\sigma$-completeness).

\begin{definition}\label{def:homo}
Let $\BM$ be a probability measure algebra.
\bce[(i)]
\item We say that $\B$ is has \emph{Maharam type} $\kappa$ if $\kappa$ is the smallest cardinal such that there is a subset $D$ of $\B$ of size $\kappa$ which \emph{completely generates} $\BMD$ (i.e.\ $\B$ is the smallest algebra containing $D$ which is closed under suprema and infima of arbitrary sets). We denote the Maharam type of $\B$ by $\tau(\B)$.

\item We say that $\B$ is \emph{Maharam-homogeneous} if  every  non-trivial principal ideal of $\BMD$ has the same Maharam type as the whole algebra, i.e.\ for every $y \neq 0$ in $\BMD$,  $\tau(\BMD) = \tau(\mathcal{B}_y)$, where $\mathcal{B}_y := \set{x \in \mathcal{B}}{x \le y}$.
\ece
\end{definition}

In the context of $\sigma$-finite measure algebras (so in particular for the probability algebras), the following Fact provides another useful characterization of Maharam-homogeneity.

\begin{Fact}
\label{f:homo}
A probability algebra $\BM$ is Maharam-homogeneous if and only if $\BMD$ is isomorphic, as a Boolean algebra, to every non-trivial principal ideal of $\BMD$ (i.e.\ to ideals $\mathcal{B}_y = \set{x \in \mathcal{B}}{x \le y}$ for some $y \in \mathcal{B}$).
\end{Fact}

\begin{proof}
See \cite[331N]{fremlin2004measure} for a proof.
\end{proof}

Maharam theorem implies that all infinite Maharam-homogeneous probability algebras are isomorphic as measure algebras to the algebra obtained from the product measure space $(2^\kappa, \Sigma_\kappa, \lambda_\kappa)$, for some infinite $\kappa$ (see Jech \cite[Example 15.31]{JECHbook} or \cite[254J]{fremlin2004measure2} for the details on this measure). Let us state the definition of the product measure space for convenience: 

\begin{definition}
Suppose $\kappa$ is an infinite cardinal and let $T$ be the set of all finite functions from $\kappa$ to $2$. Let $\Sigma_\kappa$ be the least $\sigma$-algebra generated by the system $\set{S_t}{t \in T}$, where $S_t = \set{f \in 2^\kappa}{t \sub f}$ are called the \emph{cylinder sets}. We sometimes call the collection $\Sigma_\kappa$ the \emph{Baire subsets} of $2^\kappa$. The product measure $\lambda_\kappa$ is the unique $\sigma$-additive measure on $\Sigma_\kappa$ such that $\lambda_\kappa(S_t) = \frac{1}{2^{|t|}}$.
\end{definition}

Let $\mathscr{N}_\kappa$ denote $\sigma$-ideal of sets in  $\Sigma_\kappa$ of measure zero in $\lambda_\kappa$.

\begin{definition}\label{def:m}
For an infinite cardinal $\kappa$, let $\Bk$ denote the quotient $$\Sigma_\kappa/\mathscr{N}_\kappa,$$ with naturally defined Boolean operations. It is a standard fact that $\Bk$ is a complete ccc Boolean algebra. The measure $\lambda_\kappa$ can be naturally extended to $\Bk$ by setting $$\bar{\lambda}_\kappa([x]_{\mathscr{N}_\kappa}) = \lambda_\kappa(x),$$ for $x \in \Sigma_\kappa$. The pair $(\Bk,\bar{\lambda}_\kappa)$ is a Maharam-homogenous probability measure algebra. We will often write just $\Bk$ to denote this probability algebra. We often refer to $\Bk$ as the \emph{random algebra} (of type $\kappa$).
\end{definition}

\begin{theorem}[Maharam \cite{maharam1942homogeneous}]\label{th:Mah}
Suppose $\BM$ is an infinite probability measure algebra which is Maharam-homogeneous (Definition \ref{def:homo}). Then $\BM$ is isomorphic as a measure algebra to the probability algebra $(\Bk, \bar{\lambda}_\kappa)$ for some infinite $\kappa$.
\end{theorem}

\begin{proof}
For the proof, see the main theorems in \cite[331I]{fremlin2004measure} and \cite[331L]{fremlin2004measure} .
\end{proof}

\subsection{Random algebra as a forcing notion} \label{sec:forcing}

Suppose $\kappa$ is an infinite cardinal. Let $\Sigma^+_\kappa$ denote the set $\Sigma_\kappa \setminus \mathscr{N}_\kappa$, and let $\Bk^+$ denote the set of positive elements of $\Bk$. It is easy to see that the partial order $(\Bk^+,\le)$ is densely embeddable into the pre-order $(\Sigma^+_\kappa, \le)$,\footnote{$(\Sigma^+_\kappa, \le)$ is not a partial order because $p \le q$ and $q \le p$ does not imply $p = q$. However, it does imply $[p]_{\mathscr{N}_\kappa} = [q]_{\mathscr{N}_\kappa}$. The quotient partial order on $\Sigma^+_\kappa$ mod $\mathscr{N}_\kappa$ is thus isomorphic to $(\Bk^+,\le)$.} where for $p,q \in \Sigma^+_\kappa$, $$p \le q, \mbox{ $p$ extends $q$,}\mbox{ if and only if } p \setminus q \in \mathscr{N}_\kappa.$$ For forcing purposes, we will often identify forcing with the  Boolean algebra $\Bk$ with forcing with the pre-order $(\Sigma_\kappa^+,\le)$.

\brm
If $\kappa = \omega$, then the $\sigma$-algebra of Baire sets $\Sigma_\omega$ is equal to the Borel sets of $2^\omega$. If $\kappa > \omega$, then the Borel subsets of $2^\kappa$ in the product topology form a strictly larger family: while every sets $x \in \Sigma_\kappa$ depends only on countably many cylinder sets $S_t$, the open sets in $2^\kappa$ can be made up of uncountably many cylinder sets (and hence there are open sets which are not in $\Sigma_\kappa$). A standard way of extending $\Sigma_\kappa$ to Borel sets is to use the  Carath{\'e}odory extension theorem which yields a $\sigma$-algebra of ``Lebesgue-measurable'' subsets of $2^\kappa$ (that includes all Borel sets and many other sets). See for instance \cite{fremlin2004measure1} for a general discussion of this topic. One can show, and it is also a consequence of Maharam theorem, that the Baire sets are dense in the Lebesgue-measurable sets (mod the null ideal), and hence forcing with these two partial orders is forcing-equivalent. However, the family of Baire sets $\Sigma_\kappa$ is easier to work with (see Lemma \ref{lm:abs} below).
\erm

\subsection{Absolutness for $\sigma$-distributive extensions}

Since every set in $\Sigma_\kappa$ is constructed from some cylinder sets $S_t$ by means of countably many iterations of unions and complements, the measure space $(2^\kappa, \Sigma_\kappa, \lambda_\kappa)$  (and consequently $\Bk$ as well) is absolute in a strong sense with respect to $\sigma$-distributive forcing extensions:

\begin{lemma}\label{lm:abs}
Suppose $\P$ is a $\sigma$-distributive forcing notion and $\kappa$ is an infinite cardinal. Then the random algebras $\Bk^V$ and $\Bk^{V[\P]}$ are isomorphic.
\end{lemma}

\begin{proof}
Even though $2^\kappa$ in $V[\P]$ can be strictly larger than $2^\kappa$ in $V$, the finite functions from $\kappa$ to $2$ are the same, and hence there is a natural correspondence $h$ between the cylinder sets in $V$ and $V[\P]$ mapping $S_t^V = \set{f \in (2^\kappa)^V}{t \sub f}$ to $S^{V[\P]}_t = \set{f \in (2^\kappa)^{V[\P]}}{t \sub f}$.

Every element $x$ of the least $\sigma$-algebra containing the cylinder sets can be coded by a countable sequence of ordinals in $\kappa$ (``Borel codes''), specifying the cylinder sets and the countable tree of unions and complements building $x$ from the given cylinder sets. Since $\P$ is $\sigma$-distributive, it does not add new Borel codes. It follows that the correspondence $h$ between the cylinder sets can be uniquely extended to the whole algebra $\Sigma_\kappa$, obtaining a bijection between $\Sigma_\kappa^V$ and $\Sigma_\kappa^{V[\P]}$ that preserves the measure.

It follows that the quotient Boolean algebras $\Bk =(\Sigma_\kappa/\mathscr{N}_\kappa)^V$ and $\Bk^{V[\P]} = (\Sigma_\kappa/\mathscr{N}_\kappa)^{V[\P]}$ are isomorphic.
\end{proof}

We will use Lemma \ref{lm:abs} in the analysis of the randomized Mitchell forcing (Lemma \ref{pi}).

\subsection{The quotient analysis} \label{sec:k}

For the later purposes, let us review the basic facts regarding the quotient analysis of $\Bk$. 

\begin{Fact}\label{f:iter}
Suppose $\alpha < \kappa$ are infinite cardinals. Using the fact that $\B_\alpha$ is a complete subalgebra of $\B_\kappa$, the following hold:
\bce[(i)]
\item $\Bk$ can be written as a two-stage quotient iteration of complete ccc Boolean algebras: \begin{equation}\label{q:1}
\Bk \mbox{ is forcing equivalent to } \B_\alpha * \Bk/\B_\alpha,
\end{equation} 
\item and moreover,
\begin{equation}\label{q:2}
\B_\alpha \Vdash \Bk/\B_\alpha \cong \dot{\B}^{V[\B_\alpha]}_{[\alpha,\kappa)},
\end{equation}
where $\dot{\B}^{V[\B_\alpha]}_{[\alpha,\kappa)}$ denotes the random algebra $\Bk$ as defined in the generic extension $V[\B_\alpha]$.
\ece
\end{Fact}

A general reference for this fact is Fremlin \cite[552P]{fremlin5}. Let us provide a quick review of the basic ideas in order to fix notation.

A sketch of proof of (\ref{q:1}). For every $p \in \Bk$ let $\restr{p}{\alpha} \in \B_\alpha$ denote its canonical projection to $\B_\alpha$.

Suppose $G_\alpha$ is a $\B_\alpha$-generic filter. In $V[G_\alpha]$, let $I_{G_\alpha} \sub \B_\kappa$ be the ideal of  elements of $\B_\kappa$ whose projections are incompatible with the generic filter $G_\alpha$:
\begin{equation}\label{q:Ig}
I_{G_\alpha} = \set{p \in \B_\kappa}{\restr{p}{\alpha} \not \in G_\alpha}.
\end{equation}
For every $p \in \Bk$, let us define in $V[G_\alpha]$ the equivalence class of $p$:
\begin{equation}
[p]_{G_\alpha} = \set{q \in \B_\kappa}{p \triangle q \in I_{G_\alpha}},
\end{equation}
where $\triangle$ is the symmetric difference on the Boolean algebra. The quotient forcing algebra $\B_\kappa /I_{G_\alpha}$ is composed of the equivalence classes $[p]_{G_\alpha}$.

Back in $V$, let us define a dense embedding $k: \B_\kappa \to \B_\alpha * \B_\kappa / I_{\dot{G}_\alpha}$ as follows:
\begin{equation}\label{k}
k(p) = (\restr{p}{\alpha}, [p]_{\dot{G}_\alpha}),
\end{equation}
where $[p]_{\dot{G}_\alpha}$ is a name for the equivalence class of $p$ in the Boolean quotient $\B_\kappa/I_{\dot{G}_\alpha}$. We will denote the quotient algebra $\B_\kappa / I_{\dot{G}_\alpha}$ by $\Bk/\B_\alpha$.

For (\ref{q:2}), we limit ourselves to stating that it is possible to equip the algebra $\Bk/\B_\alpha$ in $V[G_\alpha]$ by a probability measure, making it a probability measure algebra, and use Maharam theorem to argue that it is isomorphic to $\B_\kappa$ as defined in $V[G_\alpha]$. To emphasize that we deal with a quotient analysis, we denote $\B_\kappa$ in $V[G_\alpha]$ by $\B_{[\alpha,\kappa)}^{V[G_\alpha]}$.

In view of (\ref{q:2}), we identify the embedding $k$ from (\ref{q:1}) with a dense embedding from $\B_\kappa$ to $\B_\alpha * \dot{\B}_{[\alpha,\kappa)}^{V[\B_\alpha]}$:
\begin{equation}\label{q:3}
k(p) = (\restr{p}{\alpha}, [p]_{\dot{G}_\alpha}),
\end{equation}
where $[p]_{\dot{G}_\alpha}$ is viewed as an element of $\dot{\B}_{[\alpha,\kappa)}^{V[\B_\alpha]}$.

\section{Random algebras and preservation of non-special trees}

\subsection{The $\sigma$-finite chain condition}

Let us discuss a useful property of random algebras.

\begin{definition}\label{def:sfcc}
A forcing notion $\P$ is  \emph{$\sigma$-finite-cc} if there are subsets $\P_k \sub \P$ for $k < \omega$ such that:
\begin{compactenum}[(i)]
\item
$\bigcup_{k<\omega}\P_k = \P$,
\item
For every $k < \omega$ there exists a number $n_k<\omega$ such that all antichains $A \sub \P_k$ have size $\le n_k$.
\end{compactenum}
Let $\SFCC$ denote the class of all $\sigma$-finite-cc forcing notions.
\end{definition}

The class $\SFCC$ properly includes all $\sigma$-centered forcings (for which every $\P_n$ is centered) and $\sigma$-linked forcings (for which every $\P_n$ contains pairwise compatible elements). A canonical example of a forcing in $\SFCC$ (which is not $\sigma$-linked) is the probability measure algebra $\B_\kappa$ for every infinite cardinal $\kappa$.\footnote{Define $\P_0 = \emptyset$ and for $k>0$, let $\P_k$ contain all conditions $p \in \B_\kappa$ which have measure in the interval $(\frac{1}{k+1}, \frac{1}{k}]$. Then all antichains in $\P_k$ have size at most $k$.} Note that while all $\sigma$-centered forcings have size at most $2^\omega$, $\sigma$-finite-cc forcing can be arbitrarily large (as illustrated by the $\Bk$'s).

An important property of $\sigma$-finite-cc forcings is that for any $\P \in \SFCC$ and any forcing $\Q$, 
\begin{equation}\label{abs}
\Q \Vdash \P \in \SFCC.
\end{equation}
To argue for (\ref{abs}), it is easy to check that any sequence $\seq{\P_k}{k<\omega}$ from Definition \ref{def:sfcc} still has the same properties in $V[\Q]$ (with the same parameters $n_k$).

\subsection{Preservation of non-special trees and $\sigma$-finite-cc forcings}

Suppose $T$ is a tree of height $\omega_1$ without cofinal branches (of arbitrary size). 

\begin{lemma}\label{lm:finite}
Suppose there is $f: T \to \omega$ which is finite to one on chains, i.e., for every $n$, $f^{-1}{}"\{n\} \cap B$  is finite for any chain $B$. Then $T$ is special.
\end{lemma}

\begin{proof}
Define $f^*: T \to \omega \x \omega$ by induction on levels of $T$. $f^*$ will be injective on chains. Suppose $f^*$ is defined on $\restr{T}{\alpha}$. For every $t \in T_\alpha$, define 

\begin{equation}
f^*(t) = (f(t),n), 
\end{equation}

where $n = 1 + \max\set{k}{\exists s < t, f(s) = f(t) \mbox{ and } f^*(s) = (f(s),k)}$. It is easy to check that $f^*$ is well-defined and injective on chains. By composing $f^*$ with any bijection between $\omega \x \omega$ and $\omega$ one obtains a specializing function.
\end{proof}

\begin{theorem}\label{th:random1}
Suppose $\P \in \SFCC$. Suppose $T$ is a non-special tree of height $\omega_1$ without cofinal branches. Then $T$ is non-special in $V[\P]$.\footnote{An analogous result for $\sigma$-centered forcings was already observed by Stejskalova (private discussion).}
\end{theorem}

\begin{proof}
Let $\seq{\P_k}{k<\omega}$ and numbers $\seq{n_k}{k<\omega}$ be as in Definition \ref{def:sfcc}.

Suppose for contradiction $1_\P \Vdash \dot{f}: T \to \omega$ is 1-1 on chains.

Define in $V$:
\begin{equation}\label{fstar} F: T \to \omega \times \omega \end{equation} so that $F(t)$ is a pair $\la k, n\ra$ such that $k$ is the least number $k$ such that there is $p \in \P_k$ which decides $\dot{f}(t)$ and $n$ is the least $n$ which is decided as the value of $\dot{f}(t)$ by some condition in $\P_k$ (there are at most $n_k$ different natural numbers forced by elements in $\P_k$ to be equal to $\dot{f}(t)$ because every antichain in $\P_k$ has size at most $n_k$, but this is not used for the definition). 

\begin{lemma} \label{lm:f-sigma}
The function $F$ is finite to one on chains.
\end{lemma}

\begin{proof}
Suppose for contradiction that there is $\la k, n \ra$ such that $F$ is constant with value $\la k, n\ra$ on an infinite chain $B$ in $T$. We will derive a contradiction from a weaker assumption that $B$ has length $n_k+1$. Let us write $B$ as $\seq{b_i}{i<n_k+1}$, where $b_i < b_{i+1}$ for each $i$. Fix for each $i<n_k+1$ a condition $p_i \in \P_k$ such that $$p_i \Vdash \dot{f}(b_i) = n.$$ Since all $p_i$ are in $\P_k$, and all antichains in $\P_k$ have size at most $n_k$, there must be some $i < j$ such that $p_i$ and $p_j$ are compatible. But a lower bound of these conditions forces $\dot{f}(b_i) = \dot{f}(b_j)$ and yet $b_i < b_j$ in the tree order. A contradiction.
\end{proof} By Lemma \ref{lm:finite}, there is in $V$ a specializing function, and hence Theorem \ref{th:random1} is proved.
\end{proof}

\begin{corollary}
For any infinite $\kappa$, the random algebra $\B_\kappa$ does not specialize trees of height $\omega_1$ without cofinal branches.
\end{corollary}

\section{Rado's Conjecture and the ``randomized'' Mitchell forcing}

Let us show that the random algebra can be used to obtain a model of $\RC$ together with the values of the  Cicho{\' n} diagram corresponding to forcing with random algebra.

Let us start by defining a ``randomized'' Mitchell forcing. For the definition, recall that if $p \in \B_\kappa$ and $\xi < \kappa$ is a cardinal,  $\restr{p}{\xi}$ denotes a natural restriction of $p$ to $\B_\xi$ (see Section \ref{sec:k} for more details on the quotient analysis of $\Bk$).

\begin{definition}\label{def:MR}
Suppose $\kappa$ is an infinite cardinal (typically an inaccessible, but the definition is formally applicable to all cardinals). Let $\M^R_\kappa$ contain as conditions pairs $(p,q)$ such that 
\bce[(i)]
\item 
$p \in \B_\kappa$.
\item
$q$ is a function with a countable domain composed of  \emph{successor cardinals} $\xi < \kappa$ such that for all $\xi \in \mathrm{dom}(q)$, $$\B_\xi \Vdash q(\xi) \in (\dot{\Add}(\omega_1,1))^{V[\B_\xi]},$$ where $\Add(\omega_1,1)$ denotes the Cohen forcing for adding one new subset of $\omega_1$.
\ece
The ordering is defined as follows: 

$(p,q) \le (p',q') \mbox{ if and only if }p \le p' \mbox{ and } \dom{q'} \sub \dom{q} \mbox{ and } \forall \xi \in \dom{q'}$, $$\restr{p}{\xi} \Vdash q(\xi) \le q'(\xi).$$
\end{definition}

\brm \label{rm:MR}
In analogy with the quotient notation for $\B_\kappa$, which is densely embeddable into $\B_\alpha * \dot{\B}_{[\alpha,\kappa)}^{V[\B_\alpha]}$, for cardinals $\alpha < \kappa$, we apply the same notation for $\M^R_\kappa$ in preparation for Lemma \ref{pi} and Lemma \ref{th:random3}. Suppose $\alpha < \kappa$ is a cardinal (not necessarily inaccessible) and let $G_\alpha$ be $\M^R_\alpha$-generic. Then in $V[G_\alpha]$, the Mitchell forcing $\M^R_{[\alpha,\kappa)}$ denotes the forcing in Definition \ref{def:MR}, with the modification that it is defined in an obvious way with the random algebra $\B_{[\alpha,\kappa)}$ and the domains of the functions $q$ on the second coordinates are countable sets of successor cardinals in the \emph{open} interval $(\alpha,\kappa)$, i.e., $\alpha$ itself is not in the domain of $q$ (this is relevant just for Lemma \ref{th:random3}).
\erm

Suppose for the rest of the section that $\kappa$ is inaccessible. The forcing $\M^R_\kappa$ is a natural modification of the forcing introduced by Mitchell in \cite{M:tree}. Abraham \cite{ABR:tree} provided a ``product-style'' analysis of Mitchell's forcing with the Cohen forcing $\Add(\omega,\kappa)$ for adding $\kappa$ many new subsets of $\omega$, instead of $\Bk$, on the first coordinate. We will not repeat all the details of constructions in \cite{ABR:tree} here; it is easy to check that Abraham's analysis can be applied for $\M^R_\kappa$ with obvious adaptations, and yields in particular the following:

\begin{equation}\label{dist}
\M_\kappa^R \cong \B_\kappa * \dot{Q}_\kappa,
\end{equation}

for some $\sigma$-distributive forcing $\dot{Q}_\kappa$, and

\begin{equation}\label{cl}
\mbox{there is a projection onto }\M^R_\kappa \mbox{ from }\B_\kappa \x \T_\kappa,
\end{equation}
for some $\sigma$-closed forcing $\T_\kappa$ (the ``term forcing'').

It is easy to check that $\M^R_\kappa$ collapses cardinals in the interval $(\omega_1,\alpha)$ and forces $2^\omega = \alpha = \omega_2$.

The following projection lemma is crucial (see Section \ref{sec:k} for notation). It contains some new ideas related to $\Bk$ which are not present in \cite{ABR:tree} so we will give a proof. 

\begin{lemma}\label{pi}
Suppose $\alpha < \kappa$ are inaccessible cardinals. 
Then \begin{equation}\label{e}
\M^R_\kappa \mbox{ is densely embedded in }\M^R_\alpha * \dot{\M}^R_{[\alpha,\kappa)},
\end{equation}
where $\dot{\M}^R_{[\alpha,\kappa)}$ is a name for the randomized Mitchell forcing in the extension $V[\M^R_\alpha]$ (see Remark \ref{rm:MR}).
In particular, if $G_\alpha$ is $\M^R_\alpha$-generic over $V$, there is in $V[G_\alpha]$ a projection 
\begin{equation}\label{pr}\pi_\alpha: \B_{[\alpha,\kappa)}^{V[G_\alpha]} \x \T_\alpha \to \M^R_\kappa / G_\alpha\end{equation} for some $\T_\alpha$ which is $\sigma$-closed in $V[G_\alpha]$.
\end{lemma}

\begin{proof}
To prove (\ref{e}), we will define  a dense embedding $i$ which maps conditions in $\M^R_\kappa$ into conditions in $\M^R_\alpha * \dot{\M}^R_{[\alpha,\kappa)}$ in the order-preserving way. The strategy to build $i$ is very similar to \cite[Lemma 2.12]{ABR:tree}, so we will only summarize the key steps and emphasize new points. 

Recall the dense embedding $k: \Bk \to \B_\alpha * \dot{\B}_{[\alpha,\kappa)}^{V[\B_\alpha]}$ from (\ref{q:3}) in Section \ref{sec:k}. Let us write $k(p)$ as $(k^0(p),k^1(p))$.

Let us now proceed to define the dense embedding $i$. The embedding $i$ maps  a condition $(p,q)$ in $\M^R_\kappa$ to a condition $((\restr{p}{\alpha},\restr{q}{\alpha}, (\bar{p},\bar{q}))$ in $\M^R_\alpha * \dot{\M}^R_{[\alpha, \kappa)}$, where:

\medskip
\bce[(a)]
\item
$(\restr{p}{\alpha},\restr{q}{\alpha})$ is the natural restriction of $(p,q)$ to $\M^R_\alpha$.
\item \label{it:b}
The condition $k^1(p)$ is formally a $\B_\alpha$-name for a condition in $\B_{[\alpha,\kappa)}^{V[\B_\alpha]}$, but using the canonical isomorphism $h$ from Lemma \ref{lm:abs}, we can identify it with an $\M^R_\alpha$-name $k^1(p)_h$ for the same condition in $\B_{[\alpha,\kappa)}^{V[\M^R_\alpha]} \cong \B_{[\alpha,\kappa)}^{V[\B_\alpha]}$. Let $\bar{p}$ be equal to $k^1(p)_h$.
\item 
The condition $\bar{q}$ is an $\M^R_\alpha$-name for a function with its domain equal  to the domain of $q$ restricted to $[\alpha,\kappa)$. For every successor cardinal $\xi \in \dom{q} \cap [\alpha,\kappa)$, $\bar{q}(\xi)$  is an $\M^R_\alpha$-name for the $\B_{[\alpha, \xi)}^{V[\M^R_\alpha]}$-name for a condition in $\Add(\omega_1,1)^{V[\M_\alpha^R * \dot{\B}_{[\alpha, \xi)}]}$ which corresponds to the $\B_\xi$-name $q(\xi)$ for a condition in $\Add(\omega_1,1)^{V[\B_\xi]}$. This correspondence is formally defined by means of the embeddings $k$ and $h$ similarly as we used them in item (\ref{it:b}).
\ece
\medskip
Let us check that $i$ is dense.

Suppose $((r,s),a)$ is a condition in $\M^R_\alpha * \dot{\M}^R_{[\alpha,\kappa)}$. Then there is $(r',s') \le (r,s)$ which forces that $a$ is equal to a condition $(p',q')$ in $\dot{\M}^R_{[\alpha,\kappa)}$. In particular, $p'$ is an $\M^R_\alpha$-name for an element of the algebra $\B_{[\alpha,\kappa)}^{V[\M^R_\alpha]}$. Using the dense embedding $k$ from (\ref{q:3}) and the isomorphism $h$ from Lemma \ref{lm:abs}, there is a condition $p \in \B_\kappa$ such that $\restr{p}{\alpha} \le r'$ and $\restr{p}{\alpha} \Vdash k^1(p)_h \le p'$.

With regard to $q'$, we can assume without loss of generality that its domain is a ground model countable subset of successor ordinals in the interval $[\alpha,\kappa)$ (because $\B_\alpha$ is ccc and the $\sigma$-distributive quotient $\M^R_\alpha / \B_\alpha$ does not add new countable sets). Exactly as in \cite[Lemma 2.12]{ABR:tree}, using the ccc of the random algebra, we can assume that for every $\xi$ in the domain of $q'$, $q'(\xi)$ is countable ground-model name for an element of $\Add(\omega_1,1)$ in $V[\M^R_\alpha * \dot{\B}_{[\alpha,\kappa)}]$. Again using the correspondence ensured by the embedding $k,h$, we can find $q$ with the same domain as $q'$ such that $\restr{p}{\alpha}$ forces that $q$ below $\alpha$ extends $s'$ and $$i(p,q) \le ((r',s'),(p',q')),$$ as required.

Finally, the existence of the projection in (\ref{pr}) follows from (\ref{e}) by applying (\ref{dist}) in the extension $V[\M^R_\alpha]$.
\end{proof}

Let us prove the main theorem. We will not give the definitions of the cardinal invariants, but the reader can find the definitions and facts in any standard book on the subject (such as \cite{BJ:settheory}, \cite{Blasshandbook}, or \cite{JECHbook}; see also \cite{F:u} which states analogies and differences with respect to generalized cardinal invariants of higher Baire spaces $\kappa^\kappa$ and provides a useful quick summary). 

\begin{theorem}\label{th:random2}  
Suppose $\kappa$ is strongly compact. Then $\M^R_\kappa$ forces:
\bce[(i)]
\item \label{r1}
$\RC + 2^\omega = \omega_2$,
\item  \label{r3}
$\mathfrak{d}=\omega_1$, $\covN = \omega_2$, $\nonN = \omega_1$.
\ece
\end{theorem}

\begin{proof}
Regarding (\ref{r1}). We use the usual elementary embedding argument (see \cite[Section 2.1]{Zhang} for more details). Suppose $\dot{T}$ is an $\M^R_\kappa$-name for a non-special tree of height $\omega_1$ without cofinal branches. We can view $\dot{T}$ as a name for some partial order on some cardinal $\theta$. Fix an elementary embedding $j: V \to M$ witnessing that $\kappa$ is strongly compact for some $\lambda > \theta$. Suppose $$j^*: V[G] \to M[G][H_0 \times H_1]$$ is a lifted embedding in $V[G][H_0 \times H_1]$, where $G$ is $\M^R_\kappa$-generic over $V$, $H_0$ is $\B_{[\kappa,j(\kappa))}^{M[G]}$-generic over $V[G]$, and $H_1$ is $\T_\kappa^G$-generic over $V[G][H_0]$ for the term forcing which is $\sigma$-closed in $M[G]$ such that there is a projection $\pi_\kappa$ in $M[G]$:

\begin{equation}\label{projection}
\pi_\kappa:  \B_{[\kappa,j(\kappa))}^{M[G]} \times \T_\kappa \to j(\M^R_\kappa)/G.
\end{equation}

The existence of such a projection follows by Lemma \ref{pi} (and the fact that $\kappa$ is inaccessible in $M[G]$ and hence $\kappa$ is not in the domain of the functions $q$'s on the second coordinate of conditions in $j(\M^R_\kappa)$).

Let $T = \dot{T}^G$. By the properties of $j$ and $j^*$, $T$ is a subtree of $j^*(T)$ of size $<j(\kappa)$ in $M[G][H_0 \times H_1]$. We aim to show that $T$ is a non-special subtree of $j^*(T)$. It will then follow by the elementarity of $j^*$  that a non-special subtree of $T$ of size $<\kappa$ must exist in $V[G]$, thus proving $\RC$ (note that $\kappa = \omega_2$ in $V[G]$). 

It suffices to show that  $T$ is not specialized by $H_0 \x H_1$. The term forcing $\T_\kappa$ is $\sigma$-closed and hence does not specialize trees by Todorcevi{\'c} \cite{T:Rado}. In $M[G][H_1]$, the forcing $\B_{[\kappa,j(\kappa))}^{M[G]}$ is still $\SFCC$ by (\ref{abs}), and by Theorem \ref{th:random1}, does not specialize trees either.

Regarding (\ref{r3}). Let us view $\M^R_\kappa$ as $\B_\kappa * \dot{Q}_\kappa$ by (\ref{dist}). It is known that $\B_\kappa$ forces that the ground model reals remain a dominating family. Since $\dot{Q}_\kappa$ does not add new reals and collapses $(2^\omega)^V$ to $\omega_1$, the ground model reals are a dominating family of size $\omega_1$  in $V[\M^R_\kappa]$. Similarly, the ground model reals witness $\nonN = \omega_1$ in $V[\M^R_\kappa]$. Lastly, it is well-known that $\covN = \kappa$ in $V[\B_\kappa]$; it is easy to see that $\dot{Q}_\kappa$ cannot add a family $F$ of Borel null sets of size $<\kappa$ which completely cover all reals in $V[\B_\kappa]$: In some detail,  if there were some such $F \in V[\B_\kappa * \dot{Q}_\kappa]$, it would be a subset of $V[\B_\alpha]$ for some $\alpha < \kappa$ since  $\B_\kappa$ is ccc and all Borel null sets are coded by reals (via their Borel codes), and $\dot{Q}_\kappa$ does not add new reals. But any generic real added after stage $\alpha$  avoids every null set in $V[\B_\alpha]$, and in particular in $F$.\footnote{See \cite[Lemma 4.5]{HS:ureg} for more examples of invariants of $\omega^\omega$ which are not changed by $\sigma$-distributive forcing notions.}
\end{proof}

Note that a standard argument also shows that the tree property at $\omega_2$ holds in $V[\M^R_\kappa]$: Suppose $T$ is an $\omega_2$-tree in $M[G]$; the forcing $\T_\kappa$ does not add cofinal branches because it is $\sigma$-closed and $2^\omega = \omega_2$. In $M[G][H_1]$, $\B_{[\kappa,j(\kappa))}^{M[G]}$ is still $\SFCC$, and hence Knaster, and does not add cofinal branches to $T$ either. In a similar way, other tree-type compactness principles can be checked to hold in $V[\M^R_\kappa]$ as well.

\section{Open questions}

Recall the following well-known independent problems in mathematics.

\textbf{Suslin Hypothesis, $\SH$}  asserts that every dense linear order without end points which is complete and satisfies the ccc condition must be separable (and hence isomorphic to the reals). It is equivalent to the non-existence of an $\omega_1$-Suslin tree. $\SH$ follows from $\MA$ by Solovay and Tennenbaum \cite{ST:Suslin} and is falsified by $\Diamond_{\omega_1}$ (see Jensen \cite{JENfine}).

\textbf{Whitehead's Conjecture, $\WP$.} We say that an abelian group $A$ is Whitehead if every surjective homomorphism $f$ from any abelian group $B$ onto $A$ with kernel $\Z$ splits, i.e.\ there exists some homomorphism $f^*: A \to B$ such that $f \circ f^*$ is the identity on $A$. It is known that every free group is Whitehead. Whitehead asked whether the converse holds as well. Stein \cite{Stein} proved that all countable Whitehead groups are free. We write $\mathsf{WC}(\kappa)$ to assert that there exists a non-free Whitehead group of size $\kappa$ (a counterexample to all Whitehead groups being free). The question turned out to be independent from $\ZFC$. By Shelah \cite{Shelah:abelian}, $\MA$ implies $\mathsf{WC}(\kappa)$ for every regular uncountable $\kappa$ (see Eklof \cite[Section 8]{Eklof:W}), while $\Diamond_{\omega_1}(S)$ for every stationary $S$ implies $\neg \WP$ (in $V = L$, $\neg \mathsf{WC}(\kappa)$ for all regular uncountable $\kappa$). See  the Eklof's article \cite{Eklof:W} for a survey of Shelah's construction.

\textbf{Baumgartner's Axiom, $\BA$.} A set $A \sub \R$ is called $\dense$ if it has no least and greatest elements and for all $a < b$ in $A$, $A \cap (a,b)$ has size $\omega_1$. $\BA$ is the statement that all $\dense$ sets are order-isomorphic, thus extending Cantor's theorem on the categoricity of the rationals (as a linear order). $\CH$ implies the failure of $\BA$ while $\PFA$ proves $\BA$ by Baumgartner \cite{Baum:PFA} (however, the consistency strength of $\BA$ is just that of $\ZFC$ using a ccc forcing notion \cite{MR317934}).

In \cite[Section 6.3]{RH:comp} we observed that $\neg \SH$, $\neg \WP$ and $\neg \BA$ hold in standard Mitchell models yielding $\RC + 2^\omega = \omega_2$. It is easy to  observe that $\M^R_\kappa$ from Theorem \ref{th:random2} forces $\neg \SH$ and $\neg \BA$. 

\begin{lemma}\label{th:random3}
With the assumptions of Theorem \ref{th:random2}, $\M^R_\kappa$ forces  $\neg \BA$ and $\neg \SH$.\end{lemma}

\begin{proof}
By Todorcevi{\'c} \cite{T:partition}, $\BA$ implies $\mathfrak{b} > \omega_1$; since $\mathfrak{b} \le \mathfrak{d}$, having $\mathfrak{d} = \omega_1$ in $V[\M^R_\kappa]$ implies $\neg \BA$.

Regarding the existence of an $\omega_1$-Suslin tree, suppose $\alpha < \kappa$ is an infinite successor cardinal. Then by the definition of $\M^R_\kappa$ in Definition \ref{def:MR} and Remark \ref{rm:MR}, $\M^R_\kappa$ can be written as $$\M^R_\alpha * \dot{\Add}(\omega_1,1) * \dot{\M}^R_{[\alpha,\kappa)}.$$ By standard arguments, $\dot{\Add}(\omega_1,1)$ adds a diamond sequence, and hence an $\omega_1$-Suslin tree $S$. Using a product analysis of $\dot{\M}^R_{[\alpha,\kappa)}$ as in Theorem \ref{th:random2}(i), it easy to show that $S$ is preserved as a Suslin tree by $\dot{M}^R_{[\alpha,\kappa)}$.
\end{proof}

The argument for $\neg \WP$ in \cite{RH:comp} used a fact implicitly appearing in \cite{berg} that $\Add(\omega,\kappa)$ forces $\neg \WP$ to show that $\neg \WP$ holds in $V[\M_\kappa]$, where $\M_\kappa$ is a standard Mitchell forcing. It is natural to ask whether $\B_\kappa$ forces $\neg \WP$, which would yield that $\neg \WP$ holds in $V[\M^R_\kappa]$ (an adapted argument from \cite{RH:comp}).

\begin{question}\label{q}
Suppose $\kappa \ge \omega_1$. Does $\B_\kappa$ force $\neg \WP$?
\end{question}

More generally, the following was left open in \cite{RH:comp}, and remains open still:

\begin{question}\label{q1}
Are $\SH, \WP, \BA$ consistent with $\RC + 2^\omega = \omega_2$?
\end{question}

$\RC$ can be formulated for higher trees as well. Let us write $\RC(\omega_2)$ for $\RC$. Suppose $\kappa = \kappa^{<\kappa}$ is uncountable; then $\RC(\kappa^{++})$ stands for a statement that for every tree $T$ of height $\kappa^+$ without cofinal branches, if all its subtrees of $T$ of size $\le \kappa^+$ are special, so is $T$. It is known that the standard Mitchell forcing with Cohen forcing adding new subsets of $\kappa$ can be used to force $2^\kappa = \kappa^{++}$ with $\RC(\kappa^{++})$. In this model $\mf{d}_\kappa = \kappa^{++}$.  One may ask whether $\RC(\kappa^{++})$ is consistent with $\mf{d}_\kappa = \kappa^+$ (see for instance \cite{F:u} for more details on generalized cardinal invariants).

\begin{question}\label{q2}
Suppose $\kappa$ is an uncountable regular cardinal with $\kappa^{<\kappa} = \kappa$ and let us consider the generalized cardinal invariants of the space $\kappa^\kappa$. Is $\RC(\kappa^{++}) + 2^\kappa = \kappa^{++}$ consistent with $\mf{d}_\kappa = \kappa^+$?
\end{question}

To answer this question, one strategy would be to look for generalizations of random forcing to higher cardinals (see for instance Shelah \cite{Shelah:random1, Shelah:random2}) and adapt the argument in Theorem \ref{th:random2}. One could also try to prove that other forcings that preserve dominating families can be used to obtain $\RC$, such as the Sacks iteration (Zhang \cite{Zhang} showed that the countable support iteration of Sacks forcing at $\omega$  forces $\RC^B$ if iterated up to a strongly compact cardinal; however, it is open whether it forces the full version $\RC$).

\bibliographystyle{amsplain}
\bibliography{RC}

@string{IJM = "Israel J. Math."}

@string{APAL = "Ann. Pure Appl. Logic"}

@string{AML = "Annals of Mathematical Logic"}

@string{S = "Synthese"}

@string{BSL = "Bull. Symb. Log."}

@string{JLMS = "J. London Math. Soc."}

@string{AnM = "Ann. Math."}

@article {Zhang,
    AUTHOR = {Zhang, Jing},
     TITLE = {Rado's conjecture and its {B}aire version},
   JOURNAL = {J. Math. Log.},
  FJOURNAL = {Journal of Mathematical Logic},
    VOLUME = {20},
      YEAR = {2020},
    NUMBER = {1},
     PAGES = {1950015, 35},
      ISSN = {0219-0613,1793-6691},
   MRCLASS = {03E05 (03E35 03E55 03E65)},
  MRNUMBER = {4094551},
MRREVIEWER = {Luis\ Miguel\ Villegas Silva},
       DOI = {10.1142/S0219061319500156},
       URL = {https://doi.org/10.1142/S0219061319500156},
}

@article{Tdich,
author = "Todor{\v c}evi{\'c}, Stevo",
title = "Combinatorial Dichotomies in set theory",
journal = BSL,
volume = 17,
number = 1,
year = "2011",
pages = "1--72",
}

@book{T:partition,
author = "Todor{\v c}evi{\'c}, Stevo",
title = "Partition problems in topology",
series = "Contemporary Mathematics, 84",
publisher = "{American Mathematical Society, Providence, RI}",
year = 1989,
}

@article{T:Rado,
title = "{On a conjecture of R. Rado}",
author = "Todor{\v c}evi{\'c}, Stevo",
journal = JLMS,
volume = 27,
number = 1,
year = 1983,
pages = "1-8",
}

@misc{RH:comp,
title={Compactness for small cardinals in mathematics: principles, consequences, and limitations}, 
year = {2025},
author={Honzik, Radek},
note={\url{https://arxiv.org/abs/2510.27618}},
}

@article{ABR:tree,
author = "Abraham, Uri",
title = "Aronszajn trees on $\aleph_2$ and $\aleph_3$",
 JOURNAL = {Ann. Pure Appl. Logic},
volume = 24,
number = 3,
pages = "213-230",
year = 1983,
}

@article {TPW:more,
    AUTHOR = {Torres-P\'erez, V\'ictor and Wu, Liuzhen},
     TITLE = {Strong {C}hang's conjecture, semi-stationary reflection, the
              strong tree property and two-cardinal square principles},
   JOURNAL = {Fund. Math.},
  FJOURNAL = {Fundamenta Mathematicae},
    VOLUME = {236},
      YEAR = {2017},
    NUMBER = {3},
     PAGES = {247--262},
      ISSN = {0016-2736,1730-6329},
   MRCLASS = {03E05 (03E30 03E55)},
  MRNUMBER = {3600760},
MRREVIEWER = {Mohammad\ Golshani},
       DOI = {10.4064/fm257-5-2016},
       URL = {https://doi.org/10.4064/fm257-5-2016},
}

@article {TPW:cc,
    AUTHOR = {Torres-P\'erez, V\'ictor and Wu, Liuzhen},
     TITLE = {Strong {C}hang's conjecture and the tree property at
              {$\omega_2$}},
   JOURNAL = {Topology Appl.},
  FJOURNAL = {Topology and its Applications},
    VOLUME = {196},
      YEAR = {2015},
     PAGES = {999--1004},
      ISSN = {0166-8641,1879-3207},
   MRCLASS = {03E05 (03E65)},
  MRNUMBER = {3431031},
MRREVIEWER = {Xianghui\ Shi},
       DOI = {10.1016/j.topol.2015.05.061},
       URL = {https://doi.org/10.1016/j.topol.2015.05.061},
}

@article{ST:Suslin,
author = "Solovay, R. M. and Tennenbaum, S.",
title = {Iterated {Cohen Extensions and Souslin's Problem}},
year = 1971,
journal = AnM,
volume =  94,
number = 2,
pages = "201--245",
}

@article{JENfine,
author = "R. Bj{\"{o}}rn Jensen",
title = "The Fine Structure of the Constructible Hierarchy",
journal = "Annals of Mathematical Logic",
volume = 4,
number = 3,
year = 1972,
pages = "229--308",
}

@article{Stein,
author = "Stein, K.",
title = "Analytische {F}unktionen mehrerer komplexer {V}er{\"{a}}nderlichen zu vorgegebenen {P}eriodizit{\"{a}}tsmoduln und das zweite {C}ousinsche {P}roblem",
journal = "Math.\ Ann.",
volume = 123,
year = 1951,
pages = "201--222",
}

@article{Shelah:abelian,
author = "Shelah, Saharon",
title = "Infinite abelian groups, {W}hitehead problem and some constructions",
journal = IJM,
volume = 18,
pages = "243--256",
year = 1974,
doi = "https://doi.org/10.1007/BF02757281",
}

@article{Eklof:W,
author = "Eklof, Paul C.",
title = "{W}hitehead's problem is undecidable",
journal = "The American Mathematical Monthly",
volume = 83,
number = 10,
pages = "775--788",
year = 1976,
}

@incollection{Baum:PFA,
editor = "Kunen, K. and Vaughan, J. E.",
booktitle = "Handbook of set theoretic topology",
author = "Baumgartner, James E.",
title = "Aplication of the proper forcing axiom",
year = 1984,
publisher = "North-Holland Publishing Co.",
pages = "913--959",
}

@article {MR317934,
    AUTHOR = {Baumgartner, James E.},
     TITLE = {All {$\aleph \sb{1}$}-dense sets of reals can be isomorphic},
   JOURNAL = {Fund. Math.},
  FJOURNAL = {Polska Akademia Nauk. Fundamenta Mathematicae},
    VOLUME = {79},
      YEAR = {1973},
    NUMBER = {2},
     PAGES = {101--106},
      ISSN = {0016-2736,1730-6329},
   MRCLASS = {02K05 (06A05)},
  MRNUMBER = {317934},
MRREVIEWER = {John\ Hickman},
       DOI = {10.4064/fm-79-2-101-106},
       URL = {https://doi.org/10.4064/fm-79-2-101-106},
}

@article {HS:ureg,
    AUTHOR = {Honzik, Radek and Stejskalov{\'a}, {\v S}{\'a}rka},
     TITLE = {Generalized cardinal invariants for an inaccessible {$\kappa$}
              with compactness at {$\kappa^{++}$}},
   JOURNAL = {Arch. Math. Logic},
  FJOURNAL = {Archive for Mathematical Logic},
    VOLUME = {64},
      YEAR = {2025},
    NUMBER = {7-8},
     PAGES = {1077--1102},
      ISSN = {0933-5846,1432-0665},
   MRCLASS = {03E55 (03E05 03E17 03E35)},
  MRNUMBER = {4972247},
       DOI = {10.1007/s00153-025-00977-2},
       URL = {https://doi.org/10.1007/s00153-025-00977-2},
}

@book{fremlin5,
  title={Measure Theory: {V}olume 5, {S}et-theoretic measure theory, part {II}},
  author={Fremlin, David H.},
  year={2004},
  publisher={Torres Fremlin},
  address={Colchester, UK},
}

@article{maharam1942homogeneous,
  title={On homogeneous measure algebras},
  author={Maharam, Dorothy},
  journal={Proceedings of the National Academy of Sciences of the United States of America},
  volume={28},
  number={3},
  pages={108--111},
  year={1942},
  }

@book{fremlin2004measure,
  title={Measure Theory: Volume 3: Measure Algebras},
  author={Fremlin, David H.},
  year={2004},
  publisher={Torres Fremlin},
  address={Colchester, UK},
}

@book{fremlin2004measure1,
  title={Measure Theory: Volume 1: The irreducible minimum},
  author={Fremlin, David H.},
  year={2004},
  publisher={Torres Fremlin},
  address={Colchester, UK},
}

@book{fremlin2004measure2,
  title={Measure Theory: Volume 2: Broad Foundations},
  author={Fremlin, David H.},
  year={2004},
  publisher={Torres Fremlin},
  address={Colchester, UK},
}

@book{JECHbook,
author = "Jech, Tom{\'a}{\v s}",
title = "Set Theory",
subtitle = "The Third Millennium Edition, revised and expanded",
publisher = "Springer",
year = 2003,
series = "Springer Monographs in Mathematics",
address = "Berlin",
}

@unpublished{berg,
      title={Whitehead's problem and condensed mathematics}, 
      author={Jeffrey Bergfalk and Chris Lambie-Hanson and Jan Šaroch},
      year={2024},
      note={https://arxiv.org/abs/2312.09122}, 
}

@article{M:tree,
author = "Mitchell, William J.",
title = "{A}ronszajn trees and the independence of the transfer property",
journal = AML,
volume = 5,
number = 1,
pages = "21--46",
year = "1972/1973",
}

@article{F:u,
author = "Brooke-Taylor, Andrew and Fischer, Vera and Friedman, Sy-David and Montoya, Diana C.",
journal = APAL,
title = "Cardinal characteristics at $\kappa$ in a small $\mathfrak{u}(\kappa)$ model",
year = 2017,
volume = 168,
number = 1,
pages = "37--49",
}

@incollection{TODtophandbook,
editor = "Kunen, K. and Vaughan, J. E.",
booktitle = "Handbook of Set-Theoretic Topology",
author = "Todor{\v c}evi{\'c}, S.",
title = "Trees and linearly Ordered Sets",
year = 1984,
publisher = "Elsevier Science Publishers B.V.",
}

@incollection{Blasshandbook,
editor = "Foreman, Matthew and Kanamori, Akihiro",
volume = 2,
booktitle = "Handbook of Set Theory",
author = "Blass, A.",
title = "Combinatorial cardinal characteristics of the continuum",
year = 2010,
publisher = "Springer",
}

@book{BJ:settheory,
author = "Bartoszy\'nsky, T. and Judah, H.",
title = "Set theory: On the structure of the real line",
publisher = "AK Peters",
year = 1995,
}

@article {Shelah:random1,
    AUTHOR = {Shelah, Saharon},
     TITLE = {A parallel to the null ideal for inaccessible {$\lambda$}:
              {P}art {I}},
   JOURNAL = {Arch. Math. Logic},
  FJOURNAL = {Archive for Mathematical Logic},
    VOLUME = {56},
      YEAR = {2017},
    NUMBER = {3-4},
     PAGES = {319--383},
      ISSN = {0933-5846,1432-0665},
   MRCLASS = {03E35 (03E55)},
  MRNUMBER = {3633799},
MRREVIEWER = {Chris\ Lambie-Hanson},
       DOI = {10.1007/s00153-017-0524-0},
       URL = {https://doi.org/10.1007/s00153-017-0524-0},
}

@article {Shelah:random2,
    AUTHOR = {Cohen, Shani and Shelah, Saharon},
     TITLE = {Generalizing random real forcing for inaccessible cardinals},
   JOURNAL = {Israel J. Math.},
  FJOURNAL = {Israel Journal of Mathematics},
    VOLUME = {234},
      YEAR = {2019},
    NUMBER = {2},
     PAGES = {547--580},
      ISSN = {0021-2172,1565-8511},
   MRCLASS = {03E35 (03E55)},
  MRNUMBER = {4040837},
MRREVIEWER = {Miha\ E.\ Habi\v c},
       DOI = {10.1007/s11856-019-1925-z},
       URL = {https://doi.org/10.1007/s11856-019-1925-z},
}

 \end{document}